\documentclass[12pt]{article}
\usepackage{amssymb}
\usepackage{amsthm}
\usepackage{amsmath}
\usepackage{graphicx}
\usepackage{enumerate}

\DeclareMathOperator{\Max}{Max}
\DeclareMathOperator{\Min}{Min}

\newtheorem{theorem}{Theorem}
\newtheorem{definition}[theorem]{Definition}

\newtheorem{example}[theorem]{Example}
\newtheorem{corollary}[theorem]{Corollary}
\title{Residuation in finite posets}
\author{Ivan~Chajda and Helmut~L\"anger}
\date{}
\begin{document}
\footnotetext[1]{Support of the research by \"OAD, project CZ~02/2019, and support of the research of the first author by IGA, project P\v rF~2019~015, is gratefully acknowledged.}
\maketitle
\begin{abstract}
When an algebraic logic based on a poset instead of a lattice is investigated then there is a natural problem how to introduce the connective implication to be everywhere defined and satisfying (left) adjointness with the connective conjunction. We have already studied this problem for the logic of quantum mechanics which is based on an orthomodular poset or the logic of quantum effects based on a so-called effect algebra which is only partial and need not be lattice-ordered. For this, we introduced the so-called operator residuation where the values of implication and conjunction need not be elements of the underlying poset, but only certain subsets of it. However, this approach can be generalized for posets satisfying more general conditions. If these posets are even finite, we can focus on maximal or minimal elements of the corresponding subsets and the formulas for the mentioned operators can be essentially simplified. This is shown in the present paper where all theorems are explained by corresponding examples.
\end{abstract}



If an algebraic logic is based on a lattice $(L,\leq)$, we usually ask that the connective conjunction $\odot$ and the connective implication $\rightarrow$ are related by means of an adjointness, i.e.\ for all $x,y,z\in L$ we have
\[
x\odot y\leq z\text{ if and only if }x\leq y\rightarrow z.
\]
It is well-known (see e.g.\ \cite B) that for Boolean algebras one can take $x\odot y=x\wedge y$ and $x\rightarrow y=x'\vee y$ in order to obtain a residuated lattice. For modular lattices with complementation one can consider $x\odot y=y\wedge(x\vee y')$ and $x\rightarrow y=x'\vee(x\wedge y)$ in order to obtain a left-residuated lattice as shown in \cite{CF} and \cite{CL19a}. Similar results can be proved for orthomodular lattices as pointed in \cite{CL17a} and \cite{CL17b}. However, when working with posets instead of lattices, we cannot use the lattice operations and there is a problem how to introduce the operations $\odot$ and $\rightarrow$. Sometimes we define these operations by using the operators $L$ and $U$, i.e.\ the lower and upper cone, respectively, see e.g.\ \cite{CL18a}, \cite{CL19b} and \cite{CLP} for details. We call this kind of residuation an operator residuation. The disadvantage of this approach is that then the formulas for adjointness are rather huge and the corresponding sets are too large, also for finite posets. The idea of this paper is to replace these large sets for finite posets by relatively small ones which work as well.

Several papers on so-called operator residuation in posets with a unary operation were already published by the authors, see e.g.\ \cite{CL18a} -- \cite{CLP}. We investigated when operators $M$ and $R$ can be introduced such that the poset satisfies the so-called operator adjointness, i.e.\
\[
M(x,y)\subseteq L(z)\text{ if and only if }L(x)\subseteq R(y,z)
\]
where $L(a)$ denotes the lower cone of the element $a$. We were successful in several cases, for example in the case of Boolean posets and in the case of some modifications of modular or orthomodular posets. However, for the operators $M$ and $R$ we often obtained complicated formulas involving the operators $U$ and $L$ denoting the upper and lower cone, respectively. Since $R$ models the logical connective implication within the corresponding logic and a certain implication may have several values, the values of $R$ are in general not elements, but sets which may be large in concrete cases. In such a case we call the using of such an implication as unsharp reasoning, see e.g.\ \cite{CL19a} and \cite{CL19b}. On the other hand, in practice we often work with finite posets. Hence the natural question arises if in the case of finite posets the expression for $R$ can be simplified and the number of possible values of $R$ can be reduced. The present paper shows that this is indeed possible for certain finite posets, for example for finite bounded strongly modular posets, finite Boolean posets and finite relatively pseudocomplemented posets. Some examples illustrate the obtained results.

In the following, all considered structures are assumed to have a non-empty base set.

Let $(P,\leq)$ be a poset, $a\in P$ and $A,B\subseteq P$. Then we write $A\leq B$ in case $x\leq y$ for all $x\in A$ and $y\in B$. Instead of $\{a\}\leq B$ we simply write $a\leq B$. Analogously we proceed in similar cases. Further we define
\begin{align*}
L(A) & :=\{x\in P\mid x\leq A\}, \\
U(A) & :=\{x\in P\mid A\leq x\}
\end{align*}
Instead of $L(\{a,b\})$, $L(\{a\}\cup A)$, $L(A\cup B)$ and $L(U(A))$ we simply write $L(a,b)$, $L(a,A)$, $L(A,B)$ and $LU(A)$, respectively. Analogously we proceed in similar cases. Further, denote by $\Max A$ and $\Min A$ the set of all maximal and minimal elements of $A$, respectively.

For the following concepts, see e.g.\ \cite{CR}.

\begin{definition}
Let $\mathbf P=(P,\leq)$ be a poset. Then $\mathbf P$ is called {\em distributive} if it satisfies one of the following equivalent {\rm LU}-identities:
\begin{align*}
U(L(x,y),z) & \approx UL(U(x,z),U(y,z)), \\
L(U(x,y),z) & \approx LU(L(x,z),L(y,z)).
\end{align*}
The poset $\mathbf P$ is called {\em modular} {\rm(}cf.\ {\rm\cite{CR})} if for all $x,y,z\in P$ the following holds:
\[
x\leq z\text{ implies }L(U(x,y),z)=LU(x,L(y,z)).
\]
It is well-known that for lattices the modular law
\[
x\leq z\text{ implies }(x\vee y)\wedge z=x\vee(y\wedge z)
\]
can be equivalently replaced by the so-called {\em modular identity}
\[
(x\vee y)\wedge(x\vee z)\approx x\vee(y\wedge(x\vee z)).
\]
We call $\mathbf P$ {\em strongly modular} {\rm(}cf.\ {\rm\cite{CL19a})} if it satisfies the {\rm LU}-identities
\begin{align*}
L(U(x,y),U(x,z)) & \approx LU(x,L(y,U(x,z))), \\
L(U(L(x,z),y),z) & \approx LU(L(x,z),L(y,z)).
\end{align*}
\end{definition}

Every distributive poset is modular and strongly modular. A lattice is strongly modular if and only if it is modular. Concerning the {\rm LU}-conditions occurring in the definition of distributive or modular posets we note that the one-sided inclusions hold in general. Namely, because of
\[
U(x,z)\cup U(y,z)\subseteq U(L(x,y),z)
\]
we have
\[
UL(U(x,z),U(y,z))\subseteq U(L(x,y),z).
\]
Analogously,
\[
LU(L(x,z),L(y,z))\subseteq L(U(x,y),z)
\]
follows. Moreover, if $x\leq z$ then
\[
\{x\}\cup L(y,z)\subseteq L(U(x,y),z)
\]
and hence
\[
LU(x,L(y,z))\subseteq L(U(x,y),z).
\]
Let $\mathbf P=(P,\leq)$ be a poset. A unary operation $'$ on $P$ is called an {\em involution} if it satisfies the identity $x''\approx x$ and it is called {\em antitone} if $x\leq y$ implies $y'\leq x'$. If $\mathbf P$ has a least or a greatest element we will denote it by $0$ or $1$, respectively. If $\mathbf P$ has both $0$ and $1$ then it is called {\em bounded}.

In the following, instead of singletons $\{a\}$ we simply write $a$.

\begin{definition}
Let $\mathbf P=(P,\leq,{}',0,1)$ be a bounded poset with a unary operation. Then $\mathbf P$ is called {\em complemented} if it satisfies the {\rm LU}-identities $L(x,x')\approx0$ and $U(x,x')\approx1$. Moreover, $\mathbf P$ is called a {\em Boolean poset} if it is both distributive and complemented.
\end{definition}

\begin{example}\label{ex1}
The poset shown in Fig.~1

\vspace*{3mm}

\[
\setlength{\unitlength}{7mm}
\begin{picture}(6,8)
\put(3,0){\circle*{.3}}
\put(0,2){\circle*{.3}}
\put(2,2){\circle*{.3}}
\put(4,2){\circle*{.3}}
\put(6,2){\circle*{.3}}
\put(0,4){\circle*{.3}}
\put(6,4){\circle*{.3}}
\put(0,6){\circle*{.3}}
\put(2,6){\circle*{.3}}
\put(4,6){\circle*{.3}}
\put(6,6){\circle*{.3}}
\put(3,8){\circle*{.3}}
\put(3,0){\line(-3,2)3}
\put(3,0){\line(-1,2)1}
\put(3,0){\line(1,2)1}
\put(3,0){\line(3,2)3}
\put(3,8){\line(-3,-2)3}
\put(3,8){\line(-1,-2)1}
\put(3,8){\line(1,-2)1}
\put(3,8){\line(3,-2)3}
\put(0,2){\line(0,1)4}
\put(6,2){\line(0,1)4}
\put(0,4){\line(1,1)2}
\put(0,2){\line(1,1)4}
\put(2,2){\line(1,1)4}
\put(4,2){\line(1,1)2}
\put(2,2){\line(-1,1)2}
\put(4,2){\line(-1,1)4}
\put(6,2){\line(-1,1)4}
\put(6,4){\line(-1,1)2}
\put(2.85,-.75){$0$}
\put(-.6,1.9){$a$}
\put(1.2,1.9){$b$}
\put(4.45,1.9){$c$}
\put(6.4,1.9){$d$}
\put(-.6,3.9){$e$}
\put(6.4,3.9){$e'$}
\put(-.7,5.9){$d'$}
\put(1.2,5.9){$c'$}
\put(4.45,5.9){$b'$}
\put(6.4,5.9){$a'$}
\put(2.85,8.4){$1$}
\put(2.3,-1.8){{\rm Fig.~1}}
\end{picture}
\]

\vspace*{11mm}

is Boolean. Since every Boolean poset can be embedded into a Boolean algebra {\rm(}e.g.\ by means of the Dedekind-MacNeille completion{\rm)}, the complementation in Boolean posets is unique and it is even an antitone involution. For modular posets this need not be the case.
\end{example}

Now we introduce the basic notion of our study.

\begin{definition}
A {\em residuated poset} is an ordered quintuple $\mathbf P=(P,\leq,\odot,\rightarrow,1)$ where $(P,\leq,1)$ is a poset with $1$ and $\odot$ and $\rightarrow$ are mappings form $P^2$ to $2^P$ such that for all $x,y,z\in P$
\begin{enumerate}[{\rm(i)}]
\item $x\odot1\approx1\odot x\approx x$,
\item $x\odot y\approx y\odot x$,
\item $x\odot y\leq z$ if and only if $x\leq y\rightarrow z$.
\end{enumerate}
If only {\rm(i)} and {\rm(iii)} are satisfied then $\mathbf P$ is called a left-residuated poset. A {\em{\rm(}left{\rm)} residuated poset} is called {\em idempotent} if it satisfies the identity $x\odot x\approx x$. Condition {\rm(iii)} is called {\em{\rm(}left{\rm)} adjointness}.
\end{definition}

In the following we describe residuation in finite bounded posets for which neither distributivity nor modularity are assumed, and $'$ may be an arbitrary unary operation satisfying the identity $1'\approx0$ and two further very general {\rm LU}-identities.

\begin{theorem}\label{th1}
Let $(P,\leq,{}',0,1)$ be a finite bounded poset with a unary operation satisfying the {\rm(LU-)}identities
\begin{align*}
                1' & \approx0, \\
 L(U(L(x,y),y'),y) & \approx L(x,y), \\
U(L(U(x,y'),y),y') & \approx U(x,y')
\end{align*}
and define
\begin{align*}
      x\odot y & :=\Max L(U(x,y'),y), \\
x\rightarrow y & :=\Min U(L(x,y),x')
\end{align*}
for all $x,y\in B$. Then $(P,\leq,\odot,\rightarrow,1)$ is an idempotent left-residuated poset satisfying the identities
\[
x\odot0\approx0,\;0\rightarrow x\approx0',\;x\rightarrow0\approx x',\;x\rightarrow x'\approx x',\;1\rightarrow x\approx x.
\]
\end{theorem}

\begin{proof}
Let $a,b,c\in P$. Then the following are equivalent:
\begin{align*}
         a\odot b & \leq c, \\
\Max L(U(a,b'),b) & \leq c, \\
     L(U(a,b'),b) & \leq c, \\
     L(U(a,b'),b) & \subseteq L(c).
\end{align*}
Moreover, the following are equivalent:
\begin{align*}
           a & \leq b\rightarrow c, \\
           a & \leq\Min U(L(b,c),b'), \\
           a & \leq U(L(b,c),b'), \\
U(L(b,c),b') & \subseteq U(a).
\end{align*}
If $a\odot b\leq c$ then $L(U(a,b'),b)\subseteq L(c)$ and hence
\begin{align*}
U(L(b,c),b') & =U(L(b)\cap L(c),b')\subseteq U(L(b)\cap L(U(a,b'),b),b')=U(L(U(a,b'),b),b')= \\
             & =U(a,b')\subseteq U(a),
\end{align*}
i.e., $a\leq b\rightarrow c$. If, conversely, $a\leq b\rightarrow c$ then $U(L(b,c),b')\subseteq U(a)$ and hence
\begin{align*}
L(U(a,b'),b) & =L(U(a)\cap U(b'),b)\subseteq L(U(L(b,c),b')\cap U(b'),b)=L(U(L(b,c),b'),b)= \\
             & =L(b,c)\subseteq L(c),
\end{align*}
i.e., $a\odot b\leq c$. Finally,
\begin{align*}
        x\odot0 & \approx\Max L(U(x,0'),0)\approx0, \\
       x\odot x & \approx\Max L(U(x,x'),x)\approx x, \\
        x\odot1 & \approx\Max L(U(x,1'),1)\approx x, \\
       1\odot x & \approx\Max L(U(1,x'),x)\approx x, \\
0 \rightarrow x & \approx\Min U(L(0,x),0')\approx0', \\
  x\rightarrow0 & \approx\Min U(L(x,0),x')\approx x', \\
x\rightarrow x' & \approx\Min U(L(x,x'),x')\approx x', \\
 1\rightarrow x & \approx\Min U(L(1,x),1')\approx x.
\end{align*}
\end{proof}

\begin{corollary}\label{cor1}
Let $\mathbf P=(P,\leq,{}',0,1)$ be a finite bounded complemented strongly modular poset. Then $\mathbf P$ satisfies the assumptions of Theorem~\ref{th1} and hence $(P,\leq,\odot,\rightarrow,1)$, where
\begin{align*}
      x\odot y & :=\Max L(U(x,y'),y), \\
x\rightarrow y & :=\Min U(L(x,y),x')
\end{align*}
for all $x,y\in B$, is an idempotent left-residuated poset satisfying the identities
\[
x\odot0\approx0,\;0\rightarrow x\approx0',\;x\rightarrow0\approx x',\;x\rightarrow x'\approx x',\;1\rightarrow x\approx x.
\]
\end{corollary}

\begin{proof}
We have
\begin{align*}
1' & \approx L(1,1')\approx0, \\
L(U(L(x,y),y'),y) & \approx LU(L(x,y),L(y',y))\approx LU(L(x,y),0)\approx LUL(x,y)\approx L(x,y), \\
U(L(U(x,y'),y),y') & \approx ULU(y',L(y,U(y',x)))\approx UL(U(y',y),U(y',x))\approx \\
& \approx UL(1,U(x,y'))\approx ULU(x,y')\approx U(x,y').
\end{align*}
\end{proof}

Corollary~\ref{cor1} can be considered as some ``finite version'' of a theorem in \cite{CL19a}.

An example of a complemented strongly modular poset where the complementation is not an involution is shown in the following.

\begin{example}
The poset shown in Fig.~2
\[
\setlength{\unitlength}{7mm}
\begin{picture}(6,8)
\put(3,0){\circle*{.3}}
\put(0,2){\circle*{.3}}
\put(1,2){\circle*{.3}}
\put(2,2){\circle*{.3}}
\put(4,2){\circle*{.3}}
\put(6,2){\circle*{.3}}
\put(0,4){\circle*{.3}}
\put(6,4){\circle*{.3}}
\put(0,6){\circle*{.3}}
\put(2,6){\circle*{.3}}
\put(4,6){\circle*{.3}}
\put(5,6){\circle*{.3}}
\put(6,6){\circle*{.3}}
\put(3,8){\circle*{.3}}
\put(3,0){\line(-3,2)3}
\put(3,0){\line(-1,1)2}
\put(3,0){\line(-1,2)1}
\put(3,0){\line(1,2)1}
\put(3,0){\line(3,2)3}
\put(3,8){\line(-3,-2)3}
\put(3,8){\line(-1,-2)1}
\put(3,8){\line(1,-2)1}
\put(3,8){\line(1,-1)2}
\put(3,8){\line(3,-2)3}
\put(0,2){\line(0,1)4}
\put(1,2){\line(-1,2)1}
\put(1,2){\line(1,1)4}
\put(6,2){\line(0,1)4}
\put(0,4){\line(1,1)2}
\put(0,2){\line(1,1)4}
\put(2,2){\line(1,1)4}
\put(4,2){\line(1,1)2}
\put(2,2){\line(-1,1)2}
\put(4,2){\line(-1,1)4}
\put(6,2){\line(-1,1)4}
\put(6,4){\line(-1,1)2}
\put(5,6){\line(1,-2)1}
\put(2.85,-.75){$0$}
\put(-.6,1.9){$a$}
\put(.4,1.9){$f$}
\put(2.4,1.9){$b$}
\put(4.45,1.9){$c$}
\put(6.4,1.9){$d$}
\put(-.6,3.9){$e$}
\put(6.4,3.9){$e'$}
\put(-.7,5.9){$d'$}
\put(1.2,5.9){$c'$}
\put(3.35,5.9){$b'$}
\put(5.3,5.9){$g$}
\put(6.4,5.9){$a'$}
\put(2.85,8.4){$1$}
\put(2.3,-1.8){{\rm Fig.~2}}
\end{picture}
\]

\vspace*{11mm}

with
\[
\begin{array}{c|cccccccccccccc}
x  & 0 & a  & b  & c  & d  & e  & e' & d' & c' & b' & a' & f  & g \\
\hline
x' & 1 & a' & b' & c' & d' & e' & e  & d  & c  & b  & a  & a' & a
\end{array}
\]
satisfies the assumptions of Corollary~\ref{cor1}, but it is not distributive since
\[
L(U(b,p),a)=L(e,a)=L(a)\neq L(0)=LU(0,0)=LU(L(b,a),L(p,a)).
\]
\end{example}

Concerning the assumptions of Theorem~\ref{th1} we note that the one-sided inclusions hold in general. Namely,
\[
L(x,y)=LUL(x,y)=L(UL(x,y),y)\subseteq L(U(L(x,y),y'),y).
\]
Analogously, we obtain
\[
U(x,y')\subseteq U(L(U(x,y'),y),y').
\]
It is surprising that also in such a general case as described in Theorem~\ref{th1} we obtain relatively simple formulas for $\odot$ and $\rightarrow$. The following example in which the unary operation is even a complementation as well as an antitone involution shows that the assumptions of Theorem~\ref{th1} are not too restrictive.

\newpage

\begin{example}
The poset shown in Fig.~3
\[
\setlength{\unitlength}{7mm}
\begin{picture}(18,8)
\put(9,0){\circle*{.3}}
\put(6,2){\circle*{.3}}
\put(8,2){\circle*{.3}}
\put(10,2){\circle*{.3}}
\put(12,2){\circle*{.3}}
\put(6,4){\circle*{.3}}
\put(12,4){\circle*{.3}}
\put(6,6){\circle*{.3}}
\put(8,6){\circle*{.3}}
\put(10,6){\circle*{.3}}
\put(12,6){\circle*{.3}}
\put(9,8){\circle*{.3}}
\put(1,4){\circle*{.3}}
\put(17,4){\circle*{.3}}
\put(9,0){\line(-3,2)3}
\put(9,0){\line(-1,2)1}
\put(9,0){\line(1,2)1}
\put(9,0){\line(3,2)3}
\put(9,8){\line(-3,-2)3}
\put(9,8){\line(-1,-2)1}
\put(9,8){\line(1,-2)1}
\put(9,8){\line(3,-2)3}
\put(6,2){\line(0,1)4}
\put(12,2){\line(0,1)4}
\put(6,4){\line(1,1)2}
\put(6,2){\line(1,1)4}
\put(8,2){\line(1,1)4}
\put(10,2){\line(1,1)2}
\put(8,2){\line(-1,1)2}
\put(10,2){\line(-1,1)4}
\put(12,2){\line(-1,1)4}
\put(12,4){\line(-1,1)2}
\put(1,4){\line(2,-1)8}
\put(1,4){\line(2,1)8}
\put(17,4){\line(-2,-1)8}
\put(17,4){\line(-2,1)8}
\put(8.85,-.75){$0$}
\put(5.4,1.9){$a$}
\put(7.2,1.9){$b$}
\put(10.45,1.9){$c$}
\put(12.4,1.9){$d$}
\put(5.4,3.9){$e$}
\put(.4,3.9){$f$}
\put(12.4,3.9){$e'$}
\put(17.4,3.9){$f'$}
\put(5.3,5.9){$d'$}
\put(7.2,5.9){$c'$}
\put(10.45,5.9){$b'$}
\put(12.4,5.9){$a'$}
\put(8.85,8.4){$1$}
\put(8.2,-1.8){{\rm Fig.~3}}
\end{picture}
\]

\vspace*{15mm}

satisfies the assumptions of Theorem~\ref{th1}, and it is not modular since $a\leq e$, but
\[
L(U(a,f),e)=L(1,e)=L(e)\neq L(a)=LU(a)=LU(a,0)=LU(a,L(f,e)).
\]
The tables for $\odot$ and $\rightarrow$ look as follows:
\[
\begin{array}{c|cccccccccccccc}
\odot & 0 & a & b & c & d & e & f & f' & e' &    d'   &    c'   &    b'   &    a'   & 1 \\
\hline
  0   & 0 & 0 & 0 & 0 & 0 & 0 & 0 & 0  & 0  &    0    &    0    &    0    &    0    & 0 \\
  a   & 0 & a & 0 & 0 & 0 & a & f & f' & 0  &    a    &    a    &    a    &    0    & a \\
  b   & 0 & 0 & b & 0 & 0 & b & f & f' & 0  &    b    &    b    &    0    &    b    & b \\
  c   & 0 & 0 & 0 & c & 0 & 0 & f & f' & c  &    c    &    0    &    c    &    c    & c \\
  d   & 0 & 0 & 0 & 0 & d & 0 & f & f' & d  &    0    &    d    &    d    &    d    & d \\
  e   & 0 & a & b & 0 & 0 & e & f & f' & 0  &    e    &    e    &    a    &    b    & e \\
  f   & 0 & a & b & c & d & e & f & 0  & e' &    d'   &    c'   &    b'   &    a'   & f \\
  f'  & 0 & a & b & c & d & e & 0 & f' & e' &    d'   &    c'   &    b'   &    a'   & f' \\
  e'  & 0 & 0 & 0 & c & d & 0 & f & f' & e' &    c    &    d    &    e'   &    e'   & e' \\
  d'  & 0 & a & b & c & 0 & e & f & f' & c  &    d'   &    e    & \{a,c\} & \{b,c\} & d' \\
  c'  & 0 & a & b & 0 & d & e & f & f' & d  &    e    &    c'   & \{a,d\} & \{b,d\} & c' \\
  b'  & 0 & a & 0 & c & d & a & f & f' & e' & \{a,c\} & \{a,d\} &    b'   &    e'   & b' \\
  a'  & 0 & 0 & b & c & d & b & f & f' & e' & \{b,c\} & \{b,d\} &    e'   &    a'   & a' \\
  1   & 0 & a & b & c & d & e & f & f' & e' &    d'   &    c'   &    b'   &    a'   & 1
\end{array}
\]
\[
\begin{array}{c|cccccccccccccc}
\rightarrow & 0  &     a     &     b     &     c     &     d     & e  & f  & f' & e' & d' & c' & b' & a' & 1 \\
\hline
     0      & 1  &     1     &     1     &     1     &     1     & 1  & 1  & 1  & 1  & 1  & 1  & 1  & 1   & 1 \\
     a      & a' &     1     &     a'    &     a'    &     a'    & 1  & a' & a' & a' & 1  & 1  & 1  & a' & 1 \\
     b      & b' &     b'    &     1     &     b'    &     b'    & 1  & b' & b' & b' & 1  & 1  & b' & 1  & 1 \\
     c      & c' &     c'    &     c'    &     1     &     c'    & c' & c' & c' & 1  & 1  & c' & 1  & 1  & 1 \\
     d      & d' &     d'    &     d'    &     d'    &     1     & d' & d' & d' & 1  & d' & 1  & 1  & 1  & 1 \\
     e      & e' &     b'    &     a'    &     e'    &     e'    & 1  & e' & e' & e' & 1  & 1  & b' & a' & 1 \\
     f      & f' &     a'    &     b'    &     c'    &     d'    & e' & 1  & f' & e  & d  & c  & b  & a  & 1 \\
     f'     & f  &     a'    &     b'    &     c'    &     d'    & e' & f  & 1  & e  & d  & c  & b  & a  & 1 \\
     e'     & e  &     e     &     e     &     d'    &     c'    & e  & e  & e  & 1  & d' & c' & 1  & 1  & 1 \\
     d'     & d  & \{b',c'\} & \{a',c'\} &     e'    &     d     & c' & d  & d  & e' & 1  & c' & b' & a' & 1 \\
     c'     & c  & \{b',d'\} & \{a',d'\} &     c     &     e'    & d' & c  & c  & e' & d' & 1  & b' & a' & 1 \\
     b'     & b  &     e     &     b     & \{a',d'\} & \{a',c'\} & e  & b  & b  & a' & d' & c' & 1  & a' & 1 \\
     a'     & a  &     a     &     e     & \{b',d'\} & \{b',c'\} & e  & a  & a  & b' & d' & c' & b' & 1  & 1 \\
     1      & 0  &     a     &     b     &     c     &     d     & e  & f  & f' & e' & d' & c' & b' & a' & 1
\end{array}
\]
\end{example}

As a special case of Theorem~\ref{th1} we obtain

\begin{theorem}
Let $(B,\leq,{}',0,1)$ be a finite Boolean poset and define
\begin{align*}
      x\odot y & :=\Max L(x,y), \\
x\rightarrow y & :=\Min U(x',y)
\end{align*}
for all $x,y\in B$. Then $(B,\leq,\odot,\rightarrow,1)$ is an idempotent residuated poset satisfying
\begin{align*}
      x\odot y=0 & \text{ if and only if }x\leq y', \\
x\rightarrow y=1 & \text{ if and only if }x\leq y
\end{align*}
for all $x,y\in B$ and
\[
x\rightarrow0\approx x',\;x\rightarrow x'\approx x',\;1\rightarrow x\approx x.
\]
\end{theorem}

\begin{proof}
We have
\begin{align*}
 L(U(L(x,y),y'),y) & \approx L(UL(U(x,y'),U(y,y')),y)\approx L(ULU(x,y'),y)\approx L(U(x,y'),y)\approx  \\
                   & \approx LU(L(x,y),L(y',y))\approx LUL(x,y)\approx L(x,y), \\
U(L(U(x,y'),y),y') & \approx U(LU(L(x,y),L(y',y)),y')\approx U(LUL(x,y),y')\approx U(L(x,y),y')\approx \\
                   & \approx UL(U(x,y'),U(y,y'))\approx ULU(x,y')\approx U(x,y'), \\
      L(U(x,y'),y) & \approx LU(L(x,y),L(y',y))\approx LUL(x,y)\approx L(x,y), \\
			U(L(x,y),x') & \approx UL(U(x,x'),U(y,x'))\approx ULU(y,x')\approx U(x',y), \\
          x\odot y & =\Max L(x,y)=0\text{ for all }x,y\in B\text{ with }x\leq y', \\
    x\rightarrow y & =\Min U(x',y)=1\text{ for all }x,y\in B\text{ with }x\leq y, \\			
          x\odot y & \approx\Max L(x,y)\approx\Max L(y,x)\approx y\odot x.
\end{align*}
Moreover, if $x,y\in B$ and $x\odot y=0$ then
\begin{align*}
x & \in L(x)=L(x,1)=L(x,U(y,y'))=LU(L(x,y),L(x,y'))=LU(0,L(x,y'))= \\
  & =LUL(x,y')=L(x,y')\subseteq L(y')
\end{align*}
and hence $x\leq y'$. Finally, if $x,y\in B$ and $x\rightarrow y=1$ then
\begin{align*}
y & \in U(y)=U(0,y)=U(L(x,x'),y)=UL(U(x,y),U(x',y))=UL(U(x,y),1)= \\
  & =ULU(x,y)=U(x,y)\subseteq U(x)
\end{align*}
and hence $x\leq y$.
\end{proof}

Again the formulas for $\odot$ and $\rightarrow$ are very simple.

As an example of a Boolean poset, we may consider that from Example~\ref{ex1}.

\begin{example}
The tables for $\odot$ and $\rightarrow$ for the Boolean poset from Example~\ref{ex1} are as follows:
\[
\begin{array}{c|cccccccccccc}
\odot & 0 & a & b & c & d & e & e' &    d'   &    c'   &    b'   &    a'   & 1 \\
\hline
  0   & 0 & 0 & 0 & 0 & 0 & 0 & 0  &    0    &    0    &    0    &    0    & 0 \\
  a   & 0 & a & 0 & 0 & 0 & a & 0  &    a    &    a    &    a    &    0    & a \\
  b   & 0 & 0 & b & 0 & 0 & b & 0  &    b    &    b    &    0    &    b    & b \\
  c   & 0 & 0 & 0 & c & 0 & 0 & c  &    c    &    0    &    c    &    c    & c \\
  d   & 0 & 0 & 0 & 0 & d & 0 & d  &    0    &    d    &    d    &    d    & d \\
  e   & 0 & a & b & 0 & 0 & e & 0  &    e    &    e    &    a    &    b    & e \\
  e'  & 0 & 0 & 0 & c & d & 0 & e' &    c    &    d    &    e'   &    e'   & e' \\
  d'  & 0 & a & b & c & 0 & e & c  &    d'   &    e    & \{a,c\} & \{b,c\} & d' \\
  c'  & 0 & a & b & 0 & d & e & d  &    e    &    c'   & \{a,d\} & \{b,d\} & c' \\
  b'  & 0 & a & 0 & c & d & a & e' & \{a,c\} & \{a,d\} &    b'   &    e'   & b' \\
  a'  & 0 & 0 & b & c & d & b & e' & \{b,c\} & \{b,d\} &    e'   &    a'   & a' \\
1  & 0 & a & b & c & d & e & e' &    d'   &    c'   &    b'   &    a'   & 1
\end{array}
\]
\[
\begin{array}{c|cccccccccccc}
\rightarrow & 0  & a & b & c & d & e & e' & d' & c' & b' & a' & 1 \\
\hline
     0      & 1  &     1     &     1     &     1     &     1     & 1  & 1  & 1  & 1  & 1  & 1  & 1 \\
     a      & a' &     1     &     a'    &     a'    &     a'    & 1  & a' & 1  & 1  & 1  & a' & 1 \\
     b      & b' &     b'    &     1     &     b'    &     b'    & 1  & b' & 1  & 1  & b' & 1  & 1 \\
     c      & c' &     c'    &     c'    &     1     &     c'    & c' & 1  & 1  & c' & 1  & 1  & 1 \\
     d      & d' &     d'    &     d'    &     d'    &     1     & d' & 1  & d' & 1  & 1  & 1  & 1 \\
     e      & e' &     b'    &     a'    &     e'    &     e'    & 1  & e' & 1  & 1  & b' & a' & 1 \\
     e'     & e  &     e     &     e     &     d'    &     c'    & e  & 1  & d' & c' & 1  & 1  & 1 \\
     d'     & d  & \{b',c'\} & \{a',c'\} &     e'    &     d     & c' & e' & 1  & c' & b' & a' & 1 \\
     c'     & c  & \{b',d'\} & \{a',d'\} &     c     &     e'    & d' & e' & d' & 1  & b' & a' & 1 \\
     b'     & b  &     e     &     b     & \{a',d'\} & \{a',c'\} & e  & a' & d' & c' & 1  & a' & 1 \\
     a'     & a  &     a     &     e     & \{b',d'\} & \{b',c'\} & e  & b' & d' & c' & b' & 1  & 1 \\
     1      & 0  &     a     &     b     &     c     &     d     & e  & e' & d' & c' & b' & a' & 1
\end{array}
\]
As one can see, again the values of $\odot$ and $\rightarrow$ are subsets consisting of at most two elements.
\end{example}

Now, we turn our attention to a bit more complicated case where $\odot$ and $\rightarrow$ cannot be introduced as terms in $L$ and $U$.

\begin{theorem}\label{th2}
Let $(P,\leq,{}',0,1)$ be a finite bounded poset with a unary operation satisfying the {\rm(LU-)}identities
\begin{align*}
                0' & \approx1, \\
							  1' & \approx0, \\
	  						x' & \neq1\text{ for all }x\in P\setminus\{0\}, \\
 L(U(L(x,y),x'),x) & \approx LU(L(x,y),L(x',x)), \\
L(U(x',x),U(y,x')) & \approx LU(x',L(x,U(y,x'))), \\
           L(x,x') & \subseteq L(y)\text{ for all }x\in P\text{ and }y\in P\setminus\{0\}, \\
           U(x,x') & \subseteq U(y)\text{ for all }x\in P\text{ and }y\in P\setminus\{1\}
\end{align*}
and define
\[
x\odot y:=\left\{
\begin{array}{ll}
                0 & \text{if }x\leq y' \\
\Max L(U(x,y'),y) & \text{otherwise}
\end{array}
\right.,\quad x\rightarrow y:=\left\{
\begin{array}{ll}
                1 & \text{if }x\leq y \\
\Min U(L(x,y),x') & \text{otherwise}
\end{array}
\right.
\]
for all $x,y\in P$. Then $(P,\leq,\odot,\rightarrow,1)$ is a left-residuated poset satisfying
\begin{align*}
      x\odot y & =0\text{ for all }x,y\in B\text{ with }x\leq y', \\
x\rightarrow y & =1\text{ for all }x,y\in B\text{ with }x\leq y
\end{align*}
and the identities
\[
x\rightarrow0\approx x',\;1\rightarrow x\approx x.
\] 
\end{theorem}

\begin{proof}
We have
\begin{align*}
       0\odot1 & \approx1\odot0\approx0, \\
 0\rightarrow0 & \approx1\rightarrow1\approx1, \\
       x\odot1 & =\Max L(U(x,1'),1)=\Max L(x)=x\text{ for all }x\in P\setminus\{0\}, \\
      1\odot x & =\Max L(U(1,x'),x)=\Max L(x)=x\text{ for all }x\in P\setminus\{0\}, \\
 x\rightarrow0 & =\Min U(L(x,0),x')=\Min U(x')=x'\text{ for all }x\in P\setminus\{0\}, \\
1\rightarrow x & =\Min U(L(1,x),1')=x\text{ for all }x\in P\setminus\{1\}.
\end{align*}
Now let $a,b,c\in P$. If $a\leq b'$ then
\begin{align*}
a\odot b=0 & \leq c, \\
         a & \leq\Min U(L(b,c),b')=b\rightarrow c.
\end{align*}
If $b\leq c$ then
\begin{align*}
a\odot b=\Max L(U(a,b'),b) & \leq c, \\
                         a & \leq1=b\rightarrow c.
\end{align*}
For the rest of the proof assume $a\not\leq b'$ and $b\not\leq c$. Then the following are equivalent:
\begin{align*}
         a\odot b & \leq c, \\
\Max L(U(a,b'),b) & \leq c, \\
     L(U(a,b'),b) & \leq c, \\
     L(U(a,b'),b) & \subseteq L(c).
\end{align*}
Moreover, the following are equivalent:
\begin{align*}
           a & \leq b\rightarrow c, \\
           a & \leq\Min U(L(b,c),b'), \\
           a & \leq U(L(b,c),b'), \\
U(L(b,c),b') & \subseteq U(a).
\end{align*}
First assume $a\odot b\leq c$. Since $a=1$ would imply $b=1\odot b=a\odot b\leq c$, we have $a\neq1$ and hence $U(b,b')\subseteq U(a)$. Now because of $L(U(a,b'),b)\subseteq L(c)$ we obtain
\begin{align*}
U(L(b,c),b') & =U(L(b)\cap L(c),b')\subseteq U(L(b)\cap L(U(a,b'),b),b')=U(L(U(a,b'),b),b')= \\
             & =U(b',L(b,U(a,b')))=ULU(b',L(b,U(a,b')))=UL(U(b',b),U(a,b'))\subseteq \\
						 & \subseteq ULU(a)=U(a),
\end{align*}
i.e.\ $a\leq b\rightarrow c$. Finally, assume $a\leq b\rightarrow c$. Since $c=0$ would imply $a\leq b\rightarrow c=b\rightarrow0=b'$, we have $c\neq0$ and hence $L(b,b')\subseteq L(c)$. Now because of $U(L(b,c),b')\subseteq U(a)$ we obtain
\begin{align*}
L(U(a,b'),b) & =L(U(a)\cap U(b'),b)\subseteq L(U(L(b,c),b')\cap U(b'),b)=L(U(L(b,c),b'),b)= \\
             & =LU(L(b,c),L(b',b))\subseteq LUL(c)=L(c),
\end{align*}
i.e.\ $a\odot b\leq c$.
\end{proof}

Concerning the two {\rm LU}-identities occurring in Theorem~\ref{th2} we note that the one-sided inclusions hold in general. Namely,
\[
L(x,y)\cup L(x',x)\subseteq L(U(L(x,y),x'),x)
\]
and hence
\[
LU(L(x,y),L(x',x))\subseteq L(U(L(x,y),x'),x)
\]
for all $x,y\in P$. Analogously,
\[
LU(x',L(x,U(y,x')))\subseteq L(U(x',x),U(y,x'))
\]
for all $x,y\in P$.

The following example shows a poset which is not a lattice and satisfies the assumptions of Theorem~\ref{th2}.

\begin{example}\label{ex2}
The poset shown in Fig.~4

\vspace*{3mm}

\[
\setlength{\unitlength}{7mm}
\begin{picture}(6,12)
\put(3,2){\circle*{.3}}
\put(0,4){\circle*{.3}}
\put(2,4){\circle*{.3}}
\put(4,4){\circle*{.3}}
\put(6,4){\circle*{.3}}
\put(0,6){\circle*{.3}}
\put(6,6){\circle*{.3}}
\put(0,8){\circle*{.3}}
\put(2,8){\circle*{.3}}
\put(4,8){\circle*{.3}}
\put(6,8){\circle*{.3}}
\put(3,10){\circle*{.3}}
\put(3,0){\circle*{.3}}
\put(3,12){\circle*{.3}}
\put(3,2){\line(-3,2)3}
\put(3,2){\line(-1,2)1}
\put(3,2){\line(1,2)1}
\put(3,2){\line(3,2)3}
\put(3,10){\line(-3,-2)3}
\put(3,10){\line(-1,-2)1}
\put(3,10){\line(1,-2)1}
\put(3,10){\line(3,-2)3}
\put(0,4){\line(0,1)4}
\put(6,4){\line(0,1)4}
\put(0,6){\line(1,1)2}
\put(0,4){\line(1,1)4}
\put(2,4){\line(1,1)4}
\put(4,4){\line(1,1)2}
\put(2,4){\line(-1,1)2}
\put(4,4){\line(-1,1)4}
\put(6,4){\line(-1,1)4}
\put(6,6){\line(-1,1)2}
\put(3,0){\line(0,1)2}
\put(3,12){\line(0,-1)2}
\put(3.4,1.9){$f$}
\put(-.6,3.9){$a$}
\put(1.2,3.9){$b$}
\put(4.45,3.9){$c$}
\put(6.4,3.9){$d$}
\put(-.6,5.9){$e$}
\put(6.4,5.9){$e'$}
\put(-.7,7.9){$d'$}
\put(1.2,7.9){$c'$}
\put(4.45,7.9){$b'$}
\put(6.4,7.9){$a'$}
\put(3.4,9.9){$f'$}
\put(2.85,-.75){$0$}
\put(2.85,12.4){$1$}
\put(2.3,-1.8){{\rm Fig.~4}}
\end{picture}
\]

\vspace*{13mm}

satisfies the assumptions of Theorem~\ref{th2} and the tables for $\odot$ and $\rightarrow$ look as follows:
\[
\begin{array}{c|cccccccccccccc}
\odot & 0 & f & a & b & c & d & e & e' &    d'   &    c'   &    b'   &    a'   & f' & 1 \\
\hline
  0   & 0 & 0 & 0 & 0 & 0 & 0 & 0 & 0  &    0    &    0    &    0    &    0    & 0  & 0 \\
  f   & 0 & 0 & 0 & 0 & 0 & 0 & 0 & 0  &    0    &    0    &    0    &    0    & 0  & f \\
  a   & 0 & 0 & a & 0 & 0 & 0 & a & f  &    a    &    a    &    a    &    f    & a  & a \\
  b   & 0 & 0 & 0 & b & 0 & 0 & b & f  &    b    &    b    &    f    &    b    & b  & b \\
  c   & 0 & 0 & 0 & 0 & c & 0 & f & c  &    c    &    f    &    c    &    c    & c  & c \\
  d   & 0 & 0 & 0 & 0 & 0 & d & 0 & d  &    f    &    d    &    d    &    d    & d  & d \\
  e   & 0 & 0 & a & b & 0 & 0 & e & f  &    e    &    e    &    a    &    b    & e  & e \\
  e'  & 0 & 0 & 0 & 0 & c & d & f & e' &    c    &    d    &    e'   &    e'   & e' & e' \\
  d'  & 0 & 0 & a & b & c & f & e & c  &    d'   &    e    & \{a,c\} & \{b,c\} & d' & d' \\
  c'  & 0 & 0 & a & b & f & d & e & d  &    e    &    c'   & \{a,d\} & \{b,d\} & c' & c' \\
  b'  & 0 & 0 & a & f & c & d & a & e' & \{a,c\} & \{a,d\} &    b'   &    e'   & b' & b' \\
  a'  & 0 & 0 & f & b & c & d & b & e' & \{b,c\} & \{b,d\} &    e'   &    a'   & a' & a' \\
  f'  & 0 & 0 & a & b & c & d & e & e' &    d'   &    c'   &    b'   &    a'   & f' & f' \\
  1   & 0 & f & a & b & c & d & e & e' &    d'   &    c'   &    b'   &    a'   & f' & 1
\end{array}
\]
\[
\begin{array}{c|cccccccccccccc}
\rightarrow & 0  & f  &     a     &     b     &     c     &     d     & e  & e' & d' & c' & b' & a' & f' & 1 \\
\hline
     0      & 1  & 1  &     1     &     1     &     1     &     1     & 1  & 1  & 1  & 1  & 1  & 1  & 1  & 1 \\
     f      & f' & 1  &     1     &     1     &     1     &     1     & 1  & 1  & 1  & 1  & 1  & 1  & 1  & 1 \\
     a      & a' & a' &     1     &     a'    &     a'    &     a'    & 1  & a' & 1  & 1  & 1  & a' & 1  & 1 \\
     b      & b' & b' &     b'    &     1     &     b'    &     b'    & 1  & b' & 1  & 1  & b' & 1  & 1  & 1 \\
     c      & c' & c' &     c'    &     c'    &     1     &     c'    & c' & 1  & 1  & c' & 1  & 1  & 1  & 1 \\
     d      & d' & d' &     d'    &     d'    &     d'    &     1     & d' & 1  & d' & 1  & 1  & 1  & 1  & 1 \\
     e      & e' & e' &     b'    &     a'    &     e'    &     e'    & 1  & e' & 1  & 1  & b' & a' & 1  & 1 \\
     e'     & e  & e  &     e     &     e     &     d'    &     c'    & e  & 1  & d' & c' & 1  & 1  & 1  & 1 \\
     d'     & d  & d  & \{b',c'\} & \{a',c'\} &     e'    &     d     & c' & e' & 1  & c' & b' & a' & 1  & 1 \\
     c'     & c  & c  & \{b',d'\} & \{a',d'\} &     c     &     e'    & d' & e' & d' & 1  & b' & a' & 1  & 1 \\
     b'     & b  & b  &     e     &     b     & \{a',d'\} & \{a',c'\} & e  & a' & d' & c' & 1  & a' & 1  & 1 \\
     a'     & a  & a  &     a     &     e     & \{b',d'\} & \{b',c'\} & e  & b' & d' & c' & b' & 1  & 1  & 1 \\
     f'     & f  & f  &     a     &     b     &     c     &     d     & e  & e' & d' & c' & b' & a' & 1  & 1 \\
     1      & 0  & f  &     a     &     b     &     c     &     d     & e  & e' & d' & c' & b' & a' & f' & 1
\end{array}
\]
\end{example}

In Example~\ref{ex2} the unary operation $'$ is in fact an antitone involution which is not a complementation. We are going to show that $'$ even need not be an involution, see the following example.

\begin{example}
The poset shown in Fig.~5

\vspace*{-4mm}

\[
\setlength{\unitlength}{7mm}
\begin{picture}(6,11)
\put(3,2){\circle*{.3}}
\put(3,4){\circle*{.3}}
\put(1,6){\circle*{.3}}
\put(3,6){\circle*{.3}}
\put(5,6){\circle*{.3}}
\put(3,8){\circle*{.3}}
\put(3,10){\circle*{.3}}
\put(3,2){\line(0,1)8}
\put(3,4){\line(-1,1)2}
\put(3,4){\line(1,1)2}
\put(3,8){\line(-1,-1)2}
\put(3,8){\line(1,-1)2}
\put(2.875,1.25){$0$}
\put(3.4,3.85){$a$}
\put(.35,5.85){$b$}
\put(3.4,5.85){$c$}
\put(5.4,5.85){$d$}
\put(3.4,7.85){$a'$}
\put(2.85,10.4){$1$}
\put(2.2,.3){{\rm Fig.~5}}
\end{picture}
\]

\vspace*{-3mm}

with
\[
\begin{array}{c|ccccccc}
x  & 0 & a  & b & c & d & a' & 1 \\
\hline
x' & 1 & a' & c & d & b & a  & 0
\end{array}
\]
satisfies the assumptions of Theorem~\ref{th2} and the tables for $\odot$ and $\rightarrow$ look as follows:
\[
\begin{array}{c|ccccccc}
\odot & 0 & a & b & c & d & a' & 1 \\
\hline
  0   & 0 & 0 & 0 & 0 & 0 & 0  & 0 \\
  a   & 0 & 0 & 0 & 0 & 0 & 0  & a \\
  b   & 0 & 0 & b & c & 0 & b  & b \\
  c   & 0 & 0 & 0 & c & d & c  & c \\
  d   & 0 & 0 & b & 0 & d & d  & d \\
  a'  & 0 & 0 & b & c & d & a' & a' \\
  1   & 0 & a & b & c & d & a' & 1
\end{array}
\quad\quad\quad
\begin{array}{c|ccccccc}
\rightarrow & 0  & a & b & c & d & a' & 1 \\
\hline
     0      & 1  & 1 & 1 & 1 & 1 & 1  & 1 \\
     a      & a' & 1 & 1 & 1 & 1 & 1  & 1 \\
     b      & c  & c & 1 & c & c & 1  & 1 \\
     c      & d  & d & d & 1 & d & 1  & 1 \\
     d      & b  & b & b & b & 1 & 1  & 1 \\
     a'     & a  & a & b & c & d & 1  & 1 \\
     1      & 0  & a & b & c & d & a' & 1
\end{array}
\]
Since the considered poset is in fact a lattice, the values of $\odot$ and $\rightarrow$ are singletons.
\end{example}

Let us recall the concept of a relatively pseudocomplemented poset. These posets were systematically studied e.g.\ in \cite{CL18b} and \cite{CLP}. Of course, this notion is a generalization of the notion of a relatively pseudocomplemented lattice and in case of lattices both notions coincide.

\begin{definition}
A {\em poset} $(P,\leq)$ is called {\em relatively pseudocomplemented} if for every $a,b\in P$ there exists a greatest element $x$ of $P$ satisfying $L(a,x)\subseteq L(b)$. This element is called the {\em relative pseudocomplement} $a*b$ of $a$ with respect to $b$. Instead of $(P,\leq)$ we also write $(P,\leq,*)$. The binary operation $*$ on $P$ is called {\em relative pseudocomplementation}.
\end{definition}

It is worth noticing that every pseudocomplemented poset $(P,\leq,*)$ has a greatest element, namely $x*x$ ($x\in P$). For finite relatively pseudocomplemented posets the formulas for $\odot$ and $\rightarrow$ turn out to be rather simple.

\begin{theorem}
Let $(P,\leq,*)$ be a finite relatively pseudocomplemented poset and define
\begin{align*}
      x\odot y & :=\Max L(x,y), \\
x\rightarrow y & :=x*y
\end{align*}
for all $x,y\in P$. Then $(P,\leq,\odot,\rightarrow,1)$ is an idempotent residuated poset satisfying
\[
x\rightarrow y=1\text{ if and only if }x\leq y
\]
for all $x,y\in P$ and the identity $1\rightarrow x\approx x$.
\end{theorem}

\begin{proof}
For $a,b,c\in P$ the following are equivalent:
\begin{align*}
   a\odot b & \leq c, \\
\Max L(a,b) & \leq c, \\
     L(a,b) & \leq c, \\
     L(a,b) & \subseteq L(c), \\
     L(b,a) & \subseteq L(c), \\
          a & \leq b*c, \\
          a & \leq b\rightarrow c.
\end{align*}
Moreover, the following are equivalent:
\begin{align*}
a\rightarrow b & =1, \\
           a*b & =1, \\
        L(a,1) & \subseteq L(b), \\
          L(a) & \subseteq L(b), \\
             a & \leq b.
\end{align*}
Finally,
\begin{align*}
      x\odot y & \approx\Max L(x,y)\approx\Max L(y,x)\approx y\odot x, \\
      x\odot x & \approx\Max L(x,x)\approx x, \\
       x\odot1 & \approx\Max L(x,1)\approx x, \\
1\rightarrow x & \approx1*x\approx x.
\end{align*}
\end{proof}

\begin{example}
The poset shown in Fig.~6

\vspace*{-4mm}

\[
\setlength{\unitlength}{7mm}
\begin{picture}(6,9)
\put(3,2){\circle*{.3}}
\put(1,4){\circle*{.3}}
\put(5,4){\circle*{.3}}
\put(1,6){\circle*{.3}}
\put(5,6){\circle*{.3}}
\put(3,8){\circle*{.3}}
\put(3,2){\line(-1,1)2}
\put(3,2){\line(1,1)2}
\put(1,4){\line(0,1)2}
\put(1,4){\line(2,1)4}
\put(5,4){\line(-2,1)4}
\put(5,4){\line(0,1)2}
\put(3,8){\line(-1,-1)2}
\put(3,8){\line(1,-1)2}
\put(2.875,1.25){$0$}
\put(.35,3.85){$a$}
\put(5.4,3.85){$b$}
\put(.35,5.85){$c$}
\put(5.4,5.85){$d$}
\put(2.85,8.4){$1$}
\put(2.2,.3){{\rm Fig.~6}}
\end{picture}
\]

\vspace*{-3mm}

is relatively pseudocomplemented and the tables for $\odot$ and $\rightarrow$ look as follows:
\[
\begin{array}{c|cccccc}
\odot & 0 & a & b &    c    &    d    & 1 \\
\hline
  0   & 0 & 0 & 0 &    0    &    0    & 0 \\
  a   & 0 & a & 0 &    a    &    a    & a \\
  b   & 0 & 0 & b &    b    &    b    & b \\
  c   & 0 & a & b &    c    & \{a,b\} & c \\
  d   & 0 & a & b & \{a,b\} &    d    & d \\
  1   & 0 & a & b &    c    &    d    & 1
\end{array}
\quad\quad\quad
\begin{array}{c|cccccc}
\rightarrow & 0 & a & b & c & d & 1 \\
\hline
     0      & 1 & 1 & 1 & 1 & 1 & 1 \\
     a      & b & 1 & b & 1 & 1 & 1 \\
     b      & a & a & 1 & 1 & 1 & 1 \\
     c      & 0 & a & b & 1 & d & 1 \\
     d      & 0 & a & b & c & 1 & 1 \\
     1      & 0 & a & b & c & d & 1
\end{array}
\]
As one can see, here the values of $\odot$ contain at most two elements and the values of $\rightarrow$ are singletons.
\end{example}

Authors' addresses:

Ivan Chajda \\
Palack\'y University Olomouc \\
Faculty of Science \\
Department of Algebra and Geometry \\
17.\ listopadu 12 \\
771 46 Olomouc \\
Czech Republic \\
ivan.chajda@upol.cz

Helmut L\"anger \\
TU Wien \\
Faculty of Mathematics and Geoinformation \\
Institute of Discrete Mathematics and Geometry \\
Wiedner Hauptstra\ss e 8-10 \\
1040 Vienna \\
Austria, and \\
Palack\'y University Olomouc \\
Faculty of Science \\
Department of Algebra and Geometry \\
17.\ listopadu 12 \\
771 46 Olomouc \\
Czech Republic \\
helmut.laenger@tuwien.ac.at
\end{document}